\documentclass[a4paper,draft,reqno,12pt]{amsart}
\usepackage[english]{babel}
\usepackage{amsmath}
\usepackage{amssymb}
\usepackage{amscd}
\usepackage{amsthm}
\usepackage{euscript}
\usepackage{tikz}
\newtheorem{proposition}{Proposition}
\newtheorem{lemma}{Lemma}
\newtheorem{theorem}{Theorem}
\newtheorem*{theorem*}{Theorem}

\newtheorem{corollary}{Corollary}
\theoremstyle{definition}
\newtheorem{definition}{Definition}

\theoremstyle{remark}
\newtheorem {remark}{Remark}

\def\BG{{\mathbb G}}
\def\BK{{\mathbb K}}

\def\BG{{\mathbb G}}

\def\BC{{\mathbb C}}
\def\BK{{\mathbb K}}

\def\BP{{\mathbb P}}
\def\BA{{\mathbb A}}

\def\Soc{\mathrm{Soc}}

\def\mf{\mathfrak{m}}

\title[Limit points of one-parameter subgroups for additive actions]{Limit points of one-parameter subgroups for additive actions on hypersurfaces}

\thanks{The paper was supported by the grant RSF 23-71-01100}

\author{Anton Shafarevich}
\email{shafarevich.a@gmail.com}
\address{
Lomonosov Moscow State University, Faculty of Mechanics and Mathematics, Department of Higher Algebra, Leninskie Gory 1, Moscow, 119991 Russia;
\linebreak
and
\linebreak
HSE University, Faculty of Computer Science, Pokrovsky Boulevard 11, Moscow, 109028, Russia}

\begin{document}
\maketitle

\begin{abstract}
By an additive action on an algebraic variety $X$ over $\BC$, we mean an action of the group $\mathbb{G}_a^n = \mathbb{C}^n $ on $X$ with an open orbit. We study limit points of one-dimensional subgroups of $\mathbb{G}_a^n$ for additive actions on projective hypersurfaces. We say that an additive action on $X$ satisfies the \textbf{OP}-condition if for every point $x\in X $ that does not lie in the open orbit $O$ there is a point $y \in O$ and a vector $v \in \mathbb{G}_a^n$ such that $ \lim_{t\to\infty} tv\circ y = x$. We find all projective hypersurfaces on which there is an additive action satisfying the \textbf{OP}-condition. \end{abstract}

\section{Introduction}
We denote by $\mathbb{G}_a$ the additive group of the field $\mathbb{C}$ and by $\mathbb{G}_a^n$ the product of $n$ copies of groups $\mathbb{G}_a$. 

\begin{definition}
    An \textbf{additive action} on an algebraic variety $X$ over $\mathbb{C}$ is an effective action $\alpha: \mathbb{G}_a^n\times X \to X$ with an open orbit in the Zariski topology.
\end{definition}

Apparently, the problem of describing all pairs $(X, \alpha)$, where $X$ is an algebraic variety and $\alpha$ is an additive action on $X$, is too ambitious. It was shown in \cite{HT} that additive actions on the projective space $\mathbb{P}^n$ are in natural bijection with local commutative associative algebras over $\mathbb{C}$ of dimension $n+1$, and the problem of classifying local algebras is quite old and complex. Therefore, it seems natural to study additive actions that satisfy some additional natural conditions.

In \cite{AB}, matroid Schubert varieties  were introduced. These are varieties obtained in the following way. The group $\mathbb{G}_a$ is naturally embedded in $\mathbb{P}^1$ as an affine chart. Therefore, the group $\mathbb{G}_a^n$ is embedded in $(\mathbb{P}^1)^{n}$. If we consider a subspace $V \subseteq \mathbb{G}_a^n$, then the associated matroid Schubert variety is the closure of $V$ in $(\mathbb{P}^1)^n$. Schubert matroid varieties were used by Jun Huh and Botong Wang in proving the Downing-Wilson conjecture for realizable matroids; see \cite{HW}.

One can show that the action of $V$ on itself can be extended to an action of $V$ on the corresponding matroid Schubert variety, so every matroid Schubert variety admits an additive action. It can be observed that the additive action on a matroid Schubert variety satisfies the following two properties. First, the number of orbits is finite. The second property is as follows. 

\begin{definition}

We say that an additive action $\alpha$ of a group $\mathbb{G}_a^n$ on an algebraic variety $X$ satisfies the \textbf{OP}-condition (condition on one-parametric subgroups) if for every $\mathbb{G}_a^n$-orbit $O'$ in $X$ there is a point $x$ in the open orbit $O$ in $X$ and a one-dimensional algebraic subgroup $H$ in $\mathbb{G}_a^n$ such that the closure in $X$ of $H$-orbit of $x$ intersects $O'.$

\end{definition}

For an equivalent definition, see Corollary \ref{defcor}. It was shown in \cite{CC} that the additive action on a matroid Schubert variety satisfies the \textbf{OP}-condition. Moreover, the following statement is true.

 \begin{theorem*}\cite[Theorem A]{CC}
     Let $X$ be a normal complete variety and $\alpha: \mathbb{G}_a^n\times X \to X$ be an additive action. Then the following conditions are equivalent.
    \begin{enumerate}
        \item There are finitely many orbits on $X$ with respect to $\alpha$, and $\alpha$ satisfies the \textbf{OP}-condition.
        \item The variety $X$ is a matroid Schubert variety, and $\alpha$ is the corresponding additive action on~$X$.
     \end{enumerate}
 \end{theorem*}

The class of matroid Schubert  varieties is quite specific. For example, a matroid Schubert variety is smooth if and only if it is isomorphic to $(\mathbb{P}^1)^n$. Therefore, it is interesting to relax the restrictions on the additive action. Projective hypersurfaces admitting an additive action with a finite number of orbits were classified in \cite{BCS}. In \cite{Sha3} we studied complete toric varieties that admit an  action of some unipotent (not necessarily commutative) algebraic group with finitely many orbits.

In this paper we investigate additive actions satisfying the \textbf{OP}-condition. We show that on $\BP^n$ there is exactly one additive action satisfying the \textbf{OP}-condition for each $n\geq 1$; see Proposition \ref{pnop}. We also prove that if $X\subseteq \BP^{n+1}$ is a hypersurface admitting an additive action that can be extended to an action on $\BP^{n+1}$, then X is either a hyperplane or the quadric
$$X = \{2z_0z_{k+1} = z_1^2 + \ldots + z_k^2\} \subseteq \BP^{n+1}$$
where $1 \leq k \leq n$, or the cubic hypersurface
$$X = \{z_0^2z_3 - z_0z_1z_2 + \frac{z_1^3}{3} = 0 \} \subseteq \BP^{n+1} $$
where $n\geq 2;$ see Theorem \ref{maintheorem} and Corollary \ref{maincor}.

\section{Additive actions on projective spaces}

In this section we recall some facts about additive actions. By local algebra we always mean an associative commutative algebra with unity over $\mathbb{C}$ and with a unique maximal ideal. Let $A$ be a local algebra of dimension $n+1$ and $\mf$ be the maximal ideal in $A$. Then, using $A$, one can construct an additive action on $\BP^n$ in the following way. Any element $m\in \mf$ is nilpotent. Therefore, the element $\exp(m) = \sum_{i=0}^{\infty} \frac{m^i}{i!} \in A$ is well-defined. The additive group of $\mf$ is isomorphic to $\BG_a^n$, and we can define an action of the additive group of $\mf$ on $A$ as follows:

$$m\circ a = \exp(m)a,\ m \in \mf,\ a\in A.$$
Then $\mf$ acts on the projective space $\BP(A) \simeq \BP^n$:
$$m\circ [a] = [\exp(m)a],\ m \in \mf,\ a\in A\setminus\{0\}.$$
Here and thereafter by $[a]$ we denote the image of an element $a\in A\setminus\{0\}$ in $\BP(A)$. The action of $\mf$ on $\BP(A)$ described above is algebraic. 

The set $\exp{(\mf)} = \{\exp(m) \mid m\in \mf\}\subseteq A$ coincides with the set $1 + \mf$. The inclusion $\exp(\mf) \subseteq 1 + \mf$ is trivial. To see the other inclusion, we can consider the map:

$$1 +m \to \ln(1+m) = \sum_{i>0}(-1)^i\frac{m^i}{i} \in \mf$$
which is well-defined on $1 + \mf$ and $\exp(\ln(1+m)) = 1+m.$ 

It implies that the $\mf$-orbit of $[1]$ is the set $[1 + \mf]\simeq \BA^n$, which is an open affine chart in~$\BP^n$. So the action of $\mf$ on $\BP^n$ is an additive action.

\begin{theorem*}\cite[Proposition 2.15]{HT}\label{hassett}
The correspondence described above establishes a bijection between

\begin{enumerate}

     \item equivalence classes of additive actions on $\BP^n$; 
    \item isomorphism classes of local algebras of dimension $n+1$.
\end{enumerate}
\end{theorem*}

It is not difficult to describe the orbits of the additive action corresponding to the local algebra $A$.

\begin{proposition}\cite[Corollary 1.13]{AZ}\label{morb}
    Let $A$ be a local algebra of dimension $n+1$ and $\mf$ be the maximal ideal in $A$. Consider the additive action on $\BP(A)$ corresponding to $A$. Then the $\mf$-orbit of an element $[a]$ for $a\in A\setminus\{0\}$ is the set $[aA^*]$ where $A^*$ is the group of invertible elements in $A$.
\end{proposition}

Now we describe the local algebras that correspond to additive actions on $\BP^n$ satisfying the \textbf{OP}-condition. We  begin with some basic observations.

\begin{lemma}\label{limdef}
    Let $X$ be a complete variety and the group $\BG_a$ acts on $X$. Suppose that $x\in X$ is not a fixed point with respect to this action. We denote by $O$ the $\mathbb{G}_a$-orbit of~$x$. Then $\overline{O} = O\cup \{y\}$ for some $\BG_a$-fixed point $y\in X \setminus O$.
    
\end{lemma}
\begin{proof}
    Suppose that $\overline{O}$ is a normal variety. Since $O \subseteq \overline{O}$ and $O\simeq \mathbb{C}$, it follows that $\overline{O}$ is a complete normal rational curve. Therefore, $\overline{O}$ is isomorphic to $\BP^1.$ The only non-tivial action of $\BG_a$ on $\BP^1$ is given by the formula
    $$s\circ [z_0:z_1] = [z_0 : z_1 + sz_0], \ s\in \BG_a$$
    and $\BP^1$ is the union of the orbit $O$ and a fixed point. 

    If the variety $\overline{O}$ is not normal, one can consider the normalization $\pi: Z \to \overline{O}$. The action of $\BG_a$ lifts to an action on $Z$ that commutes with $\pi$. Then $Z$ is isomorphic to $\BP^1$ and consists of two $\mathbb{G}_a$-orbits, one of which is a fixed point. This implies that $\overline{O}$ consists of two $\BG_a$-orbits, one of which is a fixed point. 
\end{proof}

\begin{remark}
    In the notation of Lemma \ref{limdef}, we have $y = \lim_{s \to \infty} s\circ x.$ 

\end{remark}

\begin{corollary}\label{defcor}
    Let $X$ be a complete variety of dimension $n$, $\alpha$ an additive action of the group $\BG_a^n$ on $X$ and $O$ the open orbit. Let us fix a point $x_o \in O$. The following conditions are equivalent:

    \begin{enumerate}
        \item  the action $\alpha$ satisfies the \textbf{OP}-condition;

        \item for every orbit $O' \neq O$ there is a vector $v\in \BG_a^n$  such that $\lim_{t \to \infty} (tv\circ x_0) \in O';$

        \item for every point $y\in X\setminus O$ there is a point $x\in O$ and a vector $v\in \BG_a^n$ such that $\lim_{t\to \infty} tv\circ x = y.$
    \end{enumerate}

    Then $\alpha$ satisfies the \textbf{OP}-condition if and only if for every orbit $O' \neq O$ there is a vector $v\in \BG_a^n$  such that $\lim_{t \to \infty} (tv\circ x_0) \in O'.$
\end{corollary}
\begin{proof}
$1) \implies 2)$ Suppose $\alpha$ satisfies the \textbf{OP}-condition. Then for an orbit $O' \neq O$ there is a point $x\in O$ and a one-dimensional subgroup $H$ such that the closure of $H$-orbit of $x$ intersects $O'$. Every one-dimensional algebraic subgroup in $\BG_a^n$ is a one-dimensional subspace, so $H = \langle v \rangle$ for some non-zero $v\in \BG_a^n$. By Proposition \ref{limdef}, we have $\lim_{t\to \infty} tv \circ x = y\in O'$. There is an element $s\in \BG_a^n$ such that $s\circ x = x_0$. Since the map $z \to s\circ z$ is an automorphism of $X$ we have 
$$s\circ \overline{Hx}  = \overline{s\circ Hx}= \overline{H (s\circ x)} = \overline{Hx_0}.$$
So $\lim_{t\to \infty} tv\circ x_0 \in O'.$

2) $ \implies $ 3) Let us consider a point $y \in X \setminus O$. Then $y$ belongs to some orbit $O' \neq O$. There is a vector $v\in \BG_a^n$ such that $\lim_{t\to \infty} tv\circ x_0 = y' \in O'$. Since $y'$ and $y$ lie in the same orbit, there is $s\in \BG_a^n$ such that $y = s\circ y'$. Then

$$\lim_{t\to \infty} (tv\circ (s\circ x_0)) = s\circ \lim_{t\to \infty} (tv\circ x_0) = s\circ y' = y.$$

3)$\implies $ 1) Let us choose an orbit $O' \neq O$. Consider a point $y\in O'$. There is a point $x \in O$ and a vector $v\in \BG_a^n$ such that $\lim_{t\to \infty} tv\circ x = y$. Then $v \neq 0$, and $H = \langle v \rangle$ is a one-dimensional subgroup in $\BG_a^n$. The closure of the $H$-orbit of $x$ in the classical topology contains the point $y$. Then the closure of $H$-orbit of $x$ in the Zariski topology contains $y$.

\end{proof}
 
Now let $A$ be a local algebra of dimension $n+1$ with the maximal ideal $\mf$, and $\alpha$ be the corresponding additive action on $\BP(A) = \BP^n$. For an element $m\in \mf$ by $\lim_{t\to \infty} tm$ we will mean $\lim_{t\to \infty} tm\circ[1].$

\begin{proposition}
    Let $m$ be an element of $\mf\setminus\{0\}$. Then $\lim_{t\to \infty} tm= [m^k],$ where $k$ is the largest positive integer number such that $m^k \neq 0$.
\end{proposition}

\begin{proof}
    We have
    $$\exp(tm) = 1 + \frac{tm}{1!} + \frac{t^2m^2}{2!} + \ldots + \frac{t^km^k}{k!}.$$
    
    Since $m$ is nilpotent the vectors $1, m, m^2, \ldots, m^k$ are linearly independent. Let us complete these vectors to a basis of $A$. Then in the corresponding homogeneous coordinates on $\BP^n$ we have
    $$\exp(tm)\circ[1] = [1: \frac{tm}{1!}: \frac{t^2m^2}{2!} : \ldots : \frac{t^km^k}{k!}: 0 : \ldots : 0] = $$
    $$[\frac{1}{t^k}: \frac{m}{t^{k-1}1!}: \frac{m^2}{t^{k-2}2!}: \ldots : \frac{m^k}{k!}: 0 : \ldots : 0] 
\underset{t \to \infty}{\longrightarrow} [0 : \ldots : \frac{m^k}{k!} : 0 : \ldots : 0] = [m^k].$$

\end{proof}

\begin{proposition}\label{pnop}
    Let $\alpha$ be an additive action of $\BG_a^n$ on $\BP^n$ that satisfies the \textbf{OP}-condition. Then $\alpha$ is equivalent to the following additive action:
    \begin{equation}\label{ad1}(s_1,\ldots, s_n)\circ [z_0:z_1:\ldots :z_n] = [z_0: z_1 + s_1z_0: \ldots : z_n + s_nz_0],
    \end{equation}
    where $(s_1,\ldots, s_n) \in \BG_a^n$ and $z_0,\ldots, z_n$ are homogeneous coordinates on $\BP^n$.
\end{proposition}
    \begin{proof}
       Let us choose an element $a\in \mf\setminus \{0\}.$ Then $\mf$-orbit of the point $[a]$ is the set $[aA^*]$. Suppose there is $m_0\in \mf$ such that $\lim_{t\to \infty} tm_0 = [ca]$ for some $c\in A^*$. Then $[ca] = [m_0^k]$ and $m_0^{k+1} = 0.$ It implies that $a^2 = 0.$ So $a^2 = 0$ for all $a\in \mf.$ By \cite[Lemma 2.13]{AZ} we have $\mf^2 = 0.$ Hence, $A$ is isomorphic to the algebra $\BC[x_1,\ldots, x_n]/(x_ix_j)$ and the corresponding additive action is given by (\ref{ad1}).  
        
        The complement to the open orbit in $\BP^n$ is the set of points of the form $[m]$ with $m\in \mf\setminus\{0\}.$ Since $m^2 = 0$, we have $[m] = \lim_{t\to \infty} tm$. So this additive action satisfies the \textbf{OP}-condition.   
        \end{proof}

\begin{remark}
    The additive action \ref{ad1} has one open orbit, and the complement consists of an infinite number of fixed points. We see that the only additive action on $\mathbb{P}^n$ satisfying the \textbf{OP}-condition is "opposite" to the only additive action with a finite number of orbits.
\end{remark}

\section{Additive actions on projective hypersurfaces}

Now we turn our attention to the case of hypersurfaces. We always assume that the hypersurface is not a hyperplane. Let $X$ be a projective hypersurface in $\mathbb{P}^n$ and  $\alpha : \mathbb{G}_a^{n-1} \times X \to X$ be an additive action on $X$. We call $\alpha$ an \textbf{induced additive action} if it can be extended to an action $ \mathbb{G}_a^{n-1} \times \mathbb{P}^n \to \BP^n$.

\begin{proposition}\cite[Proposition 3]{AP}\label{APprop}

There is a one-to-one correspondence between

\begin{enumerate}\label{prophyp}
    \item equivalence classes of induced additive actions on hypersurfaces in $\BP^n$ of degree at least 2; 
    \item isomorphism classes of pairs $(A, U)$, where $A$ is a local $(n+1)$-dimensional algebra with a maximal ideal $\mathfrak{m}$ and $U$ is a $(n-1)$-dimensional subspace in $\mathfrak{m}$ that generates $A $ as an algebra with a unit. 
\end{enumerate}
\end{proposition}

Pairs $(A,U)$ as in Proposition \ref{prophyp} we call \textbf{$H$-pairs}. This correspondence is arranged as follows. If $(A,U)$ is an $H$-pair, then the additive group of $U$ acts on $\BP(A) = \BP^n$:
$$ u\circ [a] = [\exp(u)a],\ u \in U,\ a\in A.$$
We take $X$ to be the closure of the orbit of $[1]$. Then the action of $U$ preserves $X$ and defines an additive action on it. By an $H$-pair $(A,U)$, one can determine the degree and the equation of $X$.

\begin{theorem}\cite[Theorem 5.1]{AS}\label{degree} The degree of $X$ is equal to the largest number $d$ such that, $\mf^d \nsubseteq U$. 
    
\end{theorem}

\begin{theorem}\cite[Theorem 2.14]{AZ}\label{equation}
    The hypersurface $X$ is given in $\BP(A)$ by the following homogeneous equation:
    $$z_0^d\pi\left(\ln\left(1 + \frac{z}{z_0} \right)\right) = 0,$$
    where $z_0\in \BK, z\in \mf$ and $\pi:\mf \to \mf/U\simeq \BK$ is the canonical projection.
\end{theorem}    

It is also possible to determine which points are contained in $X$.

\begin{proposition}\cite[Corollary 2.18]{AZ}\label{complement}
The complement of the open $U$-orbit in $X$ is the set
$$\{ \, [a]  \mid a \in \mf \text{ such that } a^d \in U \, \},$$
where $d$ is the degree of $X$.
\end{proposition}

Among all hypersurfaces, non-degenerate hypersurfaces are of particular interest.

\begin{definition}\label{nond}
    Suppose a projective hypersurface $X \subseteq \BP^n$ is given by an equation $f(z_0, z_1,\ldots,z_n) = 0$. Then $X$ is called \emph{non-degenerate} if  there is no linear transformation of variables $z_0,\ldots, z_n$ that reduces the number of variables in $f$ to less than $n+1$.
\end{definition}

Given an $H$-pair $(A, U)$, one can determine when the corresponding hypersurface is non-degenerate. Recall that the \textbf{socle} of a local algebra $A$ is the ideal 
$$\Soc(A) = \{a\in A \mid a\mf = 0\}.$$

\begin{theorem}\cite[Theorem 2.30]{AZ}\label{Soc}
    An $H$-pair $(A,U)$ defines a non-degenerate hypersurface if and only if~$\dim(\Soc(A)) = 1$ and $\mf = U \oplus \Soc(A)$.
\end{theorem}

\begin{corollary}
Let $(A,U)$ be an $H$-pair and $X$ be the corresponding hypersurface of degree $d$. Suppose that $X$ is non-degenerate. Then for $m\in \mf$ we have $[m] \in X$ if and only if~$m^d = 0.$     
\end{corollary}
\begin{proof}
    If $m^d = 0$, then $m^d \in U$ and $[m] \in X$ by Proposition \ref{complement}. Now suppose that $[m]\in X$. Then $m^d \in U$. Consider the sequence of ideals
    $$\mf \supsetneq \mf^2 \supsetneq \ldots \supsetneq \mf^l \supsetneq \mf^{l+1} = 0.$$
    Then $\mf^l \subseteq \Soc(A)$. Since $\dim \Soc(A) =1$, we have $\mf^l = \Soc(A)$. Then $d = l$ and $m^d \in U$ if and only if $m^d = 0$.
\end{proof}

For non-degenerate hypersurfaces, it is not difficult to describe the orbit of a given point.

\begin{corollary}
    Let $(A,U)$ be an $H$-pair and $X$ be the corresponding hypersurface in $\BP(A)$. Suppose that $X$ is non-degenerate. Then the $U$-orbit of a point $x = [m]\in X, m\in \mf$ is the set $[mA^*]$.
\end{corollary}

\begin{proof}
By Proposition \ref{morb}, for every point $x = [m] \in \BP(A)$ with $m\in \mf$, the $\mf$-orbit of $x$ is $[mA^*]$. Since $X$  is non-degenerate, we have $\mf = U \oplus \Soc(A)$. Then, for all $a\in \mf$, we can write $a = u + q$, where $u \in U$ and $q \in \Soc(A)$. Thus, $\exp(a)m = \exp(u)m$. Therefore, the $\mf$-orbit of $x$ coincides with the $U$-orbit of $x$.  
\end{proof}

Now we describe the $H$-pairs corresponding to additive actions on non-degenerate hypersurfaces that satisfy the \textbf{OP}-condition.

\begin{proposition}\label{Lims}
    Let $(A,U)$ be an $H$-pair, $X$ the corresponding hypersurface of degree $d$, and $\alpha$ the corresponding additive action on $X$. Suppose that $X$ is non-degenerate. Then $\alpha$ satisfies the \textbf{OP}-condition if and only if $m = cu^k$ for all $m\in \mf$ with $\mf^d = 0$, $c\in A^*,\ u\in U$, and $u^{k+1} = 0.$
\end{proposition}

\begin{proof}
     Let us take a point $x\in X$ from the complement to the open $U$-orbit. Then $x = [m]$, where $m\in \mf$ and $m^d = 0$. The $U$-orbit of $x$ is the set $[mA^*]$. If $\alpha$ satisfies the \textbf{OP}-condition, there exist $u\in U$ and $c\in A^*$ such that 
     $$\lim_{t\to \infty} tu = [u^k] =  [mc],$$
     where $k$ is the maximal positive integer such that $u^k \neq 0$. Therefore, $m = u^kc'$ for some $c' \in A^*$ and $u^{k+1} = 0$. The inverse implication can be proven similarly.
\end{proof}

\begin{lemma}\label{lemlem}
    Let $(A, U)$ be an $H$-pair, $X$ the corresponding hypersurface in $\BP(A)$ of degree~$d$, and $\alpha $ the corresponding additive action on $X$. Suppose that $X$ is non-degenerate and $\alpha$ satisfies the \textbf{OP}-condition. Then 
    \begin{enumerate}
        \item we have $m^2 = 0$  for all $m \in \mf$ such that $m^d = 0$;
        \item the degree $d = 2$ or $3.$
    \end{enumerate}
\end{lemma}

\begin{proof}

1) Let us take $m \in \mf$ such that $m^d = 0$.  Hence, $ m = cu^k$ for some $c\in A^*$ and $u\in U$ with $u^{k+1} = 0$. It implies $m^2 = 0$. 

2) Let us take $m\in \mf^2$. Then $m^d = 0$, which implies $m^2 = 0$. By \cite[Lemma 2.13]{AZ} applied to the algebra $\BC \oplus \mf^2$ and the zero subspace, we obtain that $\mf^4 = 0$. Since $d$ is the minimal number such that $\mf^d \nsubseteq U$, we have $d\leq 3$.
\end{proof}

\begin{proposition}
    Let $(A,U)$ be an $H$-pair, $X$ the corresponding hypersurface of degree $d$, and $\alpha$ the corresponding additive action on $X$. Suppose that $X$ is non-degenerate and $\alpha$ satisfies the \textbf{OP}-condition. Then
    \begin{enumerate}
        \item if $d =2$, then 
        $$(A, U) \simeq (\mathbb{C}[x_1,\ldots, x_n]/(x_ix_j, x_i^2 - x_j^2 \mid i\neq j\ ),\ \langle x_1, \ldots, x_n \rangle);$$
        \item if $d = 3$, then 
        $$(A, U) \simeq (\mathbb{C}[x]/(x^4), \langle x, x^2\rangle).$$
    \end{enumerate}
\end{proposition}

\begin{proof}
    1) Suppose $d = 2$. We use the same arguments as in the proof of \cite[Theorem 2.25]{AZ}. We have $\mf^2 = \Soc(A)$ and $\mf^2$ is isomorphic to $\mathbb{C}$ as a vector space. Then the map 
    $$b: U\times U \to \mf^2,\ (x,y) \to xy$$
    is a symmetric bilinear form on $U$. If $u\in U$ belongs to the kernel of this form, then $u\in \mathrm{Soc}(A)$, so $u = 0,$ since $\mf = U \oplus \Soc(A)$ by Theorem \ref{Soc}. Therefore, $b$ is a non-degenerate bilinear symmetric form. Consequently, there exists a basis $x_1, \ldots, x_n \in U$ such that the matrix of $b$ in this basis is the identity matrix. It implies that $\Soc(A) = \langle x_i^2 \rangle$ and $x_ix_j = 0$ when $i\neq j$. So, we have 
    $$(A, U) \simeq (\mathbb{C}[x_1,\ldots, x_n]/(x_ix_j, x_i^2 - x_j^2 \mid i\neq j\ ),\ \langle x_1, \ldots, x_n \rangle).$$

    2) Now we consider the case $d = 3.$ In this case, $\Soc(A) = \mf^3.$ The mapping $m \to m^3$ is non-zero by \cite[Lemma 2.13]{AZ}. So there is $x \in \mf$ such that $x^3 \neq 0$. Then $x\in \mf\setminus \mf^2$ and $\mf^3 = \Soc(A) = \langle x^3 \rangle.$ Suppose that $\dim \mf/\mf^2 > 1.$ Let us choose $y_1, \ldots, y_n \in \mf$ such that the images of $x, y_1, \ldots, y_n$ in $\mf/\mf^2$ form a basis. Then $A$ is generated by $x, y_1, \ldots, y_n.$

    If $y_i^3 \neq 0$, then $y_i^3 = \lambda x^3$ for some $\lambda \in \BC\setminus \{0\}.$ Then there is $\gamma \in \BC$ such that $(y_i + \gamma x)^3 = 0.$ In this case, we replace $y_i$ with $y_i + \gamma x$ and assume that $y_i^3 = 0$ for all $i.$ By Lemma \ref{lemlem}, it implies that $y_i^2 = 0.$ Moreover, 
    $$(y_i+y_j)^3 = y_i^3 + 3y_i^2y_j + 3y_iy_j^2 + y_j^3 = 0.$$ 
    Then we have 
    $$(y_i + y_j)^2 = 2y_iy_j = 0.$$ 
    So $y_iy_j = 0$ for all $i,j.$

    Suppose $xy_i \in \langle x^2, x^3 \rangle$ for some $i$. Then $xy_i = \lambda_1x^2 + \lambda_2x^3$ for some $\lambda_1, \lambda_2  \in \BC.$ Then $$(y_i - 2\lambda_1x - 2\lambda_2x^2)^2 = 0$$ 
    but 
    $$(y_i - 2\lambda_1x - 2\lambda_2x^2)^3 = 4\lambda_1^3x^3.$$ 
    So $\lambda_1 = 0.$ Hence, $x(y_i - \lambda_2 x^2) = 0$. But $(y_i - \lambda_2 x^2)^3 = 0,$ so $(y_i - \lambda_2x^2)^2 = 0$ and $(y_i - \lambda_2x^2)y_j = 0$ for all $j\neq i.$ This means that $(y_i - \lambda_2x^2) \in \Soc(A)$. The element $y_i - \lambda_2x^2$ is linearly independent with $x_i^3$ so $\dim \Soc(A)>1$. This leads to a contradiction.

    Then $xy_i$ is linearly independent of $x^2$ and $x^3.$ The element $x^2y_i$ belongs to $\mf^3$, so $x^2y_i = \lambda_3x^3$. If $\lambda_3 \neq 0$, then $(x - \frac{1}{3\lambda_3}y_i)^3 = 0$. But $(x - \frac{1}{3\lambda_3}y_i)^2 = x^2 - 2\frac{1}{3\lambda_3}xy_i \neq 0$. This contradicts Lemma \ref{lemlem}. So $x^2y_i = 0$. But then $xy_i \in \Soc(A)$, which is a contradiction. 

    Therefore, $A$ is generated by $x$. Then $A \simeq \BC[x]/(x^n).$ Since $d = 3$, we obtain that $n = 4$, and up to an automorphism of $A$, we can assume that $U = \langle x, x^2 \rangle.$

    According to Proposition \ref{Lims}, additive actions on hypersurfaces corresponding to the both $H$-pairs satisfy the \textbf{OP}-condition.

\end{proof}

Using Proposition \ref{equation}, one can determine the equations that define the hypersurfaces.

\begin{theorem}\label{maintheorem}
    Let $X$ be a non-degenerate hypersurface and $\alpha$ an induced additive action satisfying the \textbf{OP}-condition. Then $X$ is either a non-degenerate quadric
    $$X = \{2z_0z_{n+1} =  z_1^2 + \ldots + z_n^2\} \subseteq \BP^{n+1} $$
    with $n\geq 1$, and the action $\alpha$ is given by the formula
    $$(s_1,\ldots, s_n) \circ [z_0:z_1:\ldots:z_{n+1}] = [z_0:z_1 + s_1z_0:\ldots :z_n + s_nz_0: z_{n+1} + \sum_{i=1}^n s_iz_i + \frac{\sum_{i=1}^n s_i^2}{2}z_0],$$
    or $X$ is given by the equation
    $$X = \{z_0^2z_3 - z_0z_1z_2 + \frac{z_1^3}{3} = 0 \} \subseteq \BP^3 $$
    and $\alpha$ is given by the formula
    $$(s_1, s_2) \circ[z_0:z_1:z_2:z_3] = $$ 
    $$[z_0 : z_1 + s_1z_0 : z_2 + s_1z_1 + (s_2 + 
\frac{s_1^2}{2})z_0: z_3 + s_1z_2 + (s_2 + \frac{s_1^2}{2})z_1 + (s_1s_2 + \frac{s_1^3}{6})z_0].$$    

\end{theorem}

Non-degenerate quadrics are smooth. In fact, they are the only smooth hypersurfaces with an additive action; see \cite[Corollary 2.17]{AZ}. There are infinitely many $\mathbb{G}_a^n$-orbits on non-degenerate quadrics for $n\geq 3$ in the notation of Theorem \ref{maintheorem}. Among them, one is open and consists of points with $z_0 \neq 0$. The other one is the fixed point $[0: 0 : \ldots : 0 : 1]$. Additionally, there are infinitely many one-dimensional orbits of the form
$$\{[0: z_1 : \ldots : z_n : *] \},$$
where $z_1,\ldots, z_n$ are fixed, and by $*$ we mean an arbitrary complex number. When we choose $v = (s_1, \ldots, s_n)\in \BG_a^n$ with $\sum_i s_i^2 \neq 0$, we have
$$\lim_{t\to \infty} tv\circ x_0 = [0:\ldots: 0  : 1],$$
where $x_0 = [1:0:\ldots : 0]$. If we choose non-zero $ v = (s_1, \ldots, s_n) \in \mathbb{G}_a^n$ with $\sum_is_i^2 = 0$ then
$$\lim_{t\to \infty} tv\circ x_0 = [0: s_1 : s_2 : \ldots : s_n : \sum_i s_i].$$

The hypersurface $\{z_0^2z_3 - z_0z_1z_2 + \frac{z_1^3}{3} = 0 \} \subseteq \BP^3 $ is not smooth and is not even normal; see \cite[Proposition 3]{ABZ}. It has three $\mathbb{G}_a^2$-orbits. One orbit is open, consisting of points with $z_0 \neq 0$. One is a line $\{ [0 : 0 : 1 : *]\}$, and the last one is a point $[0 : 0 : 0 : 1]$. When we choose $v = (s_1, s_2)$ with $s_1 \neq 0$, we have
$$\lim_{t\to \infty} tv \circ x_0 = [0:0:0:1],$$
where $x_0 = [1:0:0:0]$. If we take $v = (0, 1)$, then
$$\lim_{t\to \infty} tv \circ x_0 = [0:0:1:0].$$

\section*{Degenerate hypersurfaces}

Now let us consider degenerate hypersurfaces. If $(A, U)$ is an $H$-pair corresponding to a degenerate hypersurface, then $W = U \cap \Soc(A) \neq 0$ is an ideal in $A$. The pair $(A/W, U/W)$ is an $H$-pair corresponding to a hypersurface which can be given by the same equation but in a projective space of smaller dimension; see \cite[Proposition 2.20]{AZ}.

\begin{proposition}
    Let $(A, U)$ be an $H$-pair, $X$ be the corresponding hypersurface, and $\alpha$ be the corresponding additive action on $X$. Suppose that $\alpha$ satisfies the \textbf{OP}-condition. We denote by $W$ the ideal $U \cap \Soc(A)$. Then the $H$-pair $(A/W, U/W)$ corresponds to a hypersurface $Y$, and the corresponding additive action $\beta$ on~$Y$ satisfies the \textbf{OP}-condition.
\end{proposition}

\begin{proof}
The maximal ideal in $A/W$ is the ideal $\overline{\mf} = \mf + W$. Let us denote by $d$ the degree of $X$. Then the degree of $Y$ is also $d$. Consider an element $\overline{m} = m + W \in \overline{\mf}$, where $m\in \mf$. Suppose $[\overline{m}] \in Y $. Then $\overline{m}^d \in U/W$. This implies $m^d \in U$. Therefore, we have $[m] \in X$.

The $U$-orbit of $[m]$ is the set $[\exp(U)m]$. Since $\alpha$ satisfies the \textbf{OP}-condition, there is $u \in U$ such that $[u^k] = [cm]$ where $c = \exp(u')$ for some $u'\in U$ and $u^{k+1} = 0$. But then $[\overline{u}^k] = [\overline{c}\overline{m}],$ where $\overline{u}  = u + W,\ \overline{c} = c + W = \exp(u' + W)$, and $\overline{u}^{k+1} = 0$. Therefore, $\beta$ satisfies the \textbf{OP}-condition.

\end{proof}

\begin{corollary}\label{maincor}
    Let $X \subseteq \BP^{n+1}$ be a projective hypersurface. Suppose that there is an induced additive action $\alpha$ on $X$ which satisfies the \textbf{OP}-condition. Then, in some coordinates $z_0, \ldots, z_{n+1}$ on $\BP^{n+1}$, the hypersurface $X$ is given by one of the following equations:

    $$X = \{2z_0z_{k+1} =  z_1^2 + \ldots + z_k^2\} \subseteq \BP^{n+1} $$
    for some $1 \leq k \leq n$, or

        $$X = \{z_0^2z_3 - z_0z_1z_2 + \frac{z_1^3}{3} = 0 \} \subseteq \BP^{n+1} $$
    for $n \geq 2.$
    
\end{corollary}

On each of the above hypersurfaces, there is an additive action satisfying the \textbf{OP}-condition. If 

$$X = \{2z_0z_{k+1} =  z_1^2 + \ldots + z_k^2\} \subseteq \BP^{n+1},$$
then there is an additive action
$$(s_1,\ldots, s_k, \ldots, s_n) \circ [z_0:z_1:\ldots:z_{n+1}] = $$
$$[z_0:z_1 + s_1z_0:\ldots :z_k + s_kz_0: z_{k+1} + \sum_{i=1}^k s_iz_i + \frac{\sum_{i=1}^k s_i^2}{2}z_0: z_{k+2} + s_{k+1}z_0: \ldots : z_{n+1} + s_nz_0],$$
and if 
$$X = \{z_0^2z_3 - z_0z_1z_2 + \frac{z_1^3}{3} = 0 \} \subseteq \BP^{n+1},$$
then, similarly, there is an additive action
$$(s_1,s_2, \ldots, s_n) \circ [z_0:z_1:\ldots:z_{n+1}] = $$
 $$[z_0 : z_1 + s_1z_0 : z_2 + s_1z_1 + (s_2 + 
\frac{s_1^2}{2})z_0: z_3 + s_1z_2 + (s_2 + \frac{s_1^2}{2})z_1 + (s_1s_2 + \frac{s_1^3}{6})z_0: $$
$$:z_4 + s_3z_0:\ldots :z_{n+1} + s_nz_0].$$    

According \cite[Theorem 2.2]{Be}, there are at least two induced additive actions on these hypersurfaces. However, other additive actions may not satisfy the \textbf{OP}-condition. For example, on the hypersurface

$$X = \{2z_2z_2 = z_1^2 \} \subseteq \BP^{3} $$
there is an additive action
$$(s_1, s_2) \circ [z_0:z_1:z_2:z_3] = [z_0 : z_1 + s_1z_0 : z_2 + s_1z_1 + \frac{s_1^2}{2}z_0 : z_3 + s_1z_2 + \frac{s_1^2}{2}z_1 + (s_2 + \frac{s_1^3}{6})z_0].$$
But $\lim_{t\to \infty} tv\circ x_0 = [0:0:0:1]$ for all $v \in \BG_a^2\setminus\{0\}$ where $x_0 = [1:0:0:0].$

\end{document}